\documentclass[final]{siamltex}

\usepackage[usenames]{color}
\usepackage[color]{showkeys}
\definecolor{refkey}{gray}{0.4}
\definecolor{labelkey}{gray}{0.3}

\usepackage[left=1in,top=1in,right=1in,bottom=1in,dvips,letterpaper]{geometry}
\usepackage{setspace,url}
\usepackage[leqno]{amsmath}
\usepackage{amsxtra, amsfonts,amscd,  amssymb, graphicx,subfigure}

\usepackage[algo2e, vlined,ruled]{algorithm2e}
\numberwithin{equation}{section}

\newcommand{\be}{\begin{equation}}
\newcommand{\ee}{\end{equation}}
\newcommand{\ba}{\begin{array}}
\newcommand{\ea}{\end{array}}
\newcommand{\bea}{\begin{eqnarray}}
\newcommand{\eea}{\end{eqnarray}}
\newcommand{\beaa}{\begin{eqnarray*}}
\newcommand{\eeaa}{\end{eqnarray*}}

\newcommand{\half}{\frac{1}{2}}
\newcommand{\br}{\mathbb{R}}

\usepackage[usenames]{color}
\usepackage[normalem]{ulem}

\newcommand{\LCal}{\mathcal{L}}
\newcommand{\ACal}{\mathcal{A}}
\newcommand{\BCal}{\mathcal{B}}
\newcommand{\CCal}{\mathcal{C}}

\newcommand{\PCal}{\mathcal{P}}

\newcommand{\Shrink}{\mathrm{Shrink}}

\newcommand{\sgn}{\mathrm{sgn}}

\newcommand{\Xvec}{\mathrm{vec}}

\newcommand{\Tr}{\mathbf{Tr}}

\newcommand{\etal}{{et al. }}

\newcommand{\st}{\mbox{ s.t. }}

\newcommand{\prox}{\mathrm{prox}}

\newtheorem{remark}[theorem]{Remark}

\begin{document}

\title{Alternating Direction Method of Multipliers for Sparse Principal Component Analysis}
\date{\today}
\author{Shiqian Ma\footnotemark[1]}
\renewcommand{\thefootnote}{\fnsymbol{footnote}}
\footnotetext[1]{Institute for Mathematics and Its Applications, 400 Lind Hall, 207 Church Street SE, University of Minnesota, Minneapolis, MN 55455, USA.
\quad Email: maxxa007@ima.umn.edu.}

\renewcommand{\thefootnote}{\arabic{footnote}}

\maketitle \centerline{November 28, 2011}

\begin{abstract} We consider a convex relaxation  of sparse principal component analysis proposed by d'Aspremont \etal
in \cite{daspremont-sparsePCA-direct-formulation-2007}. This convex relaxation is a nonsmooth semidefinite programming problem in which the $\ell_1$ norm
of the desired matrix is imposed in either the objective function or
the constraint to improve the sparsity of the resulting matrix. The sparse principal component is obtained by a rank-one decomposition of the resulting
sparse matrix. We propose an alternating direction method based on a variable-splitting technique and an augmented Lagrangian framework for solving this
nonsmooth semidefinite programming problem. In contrast to the first-order method proposed in \cite{daspremont-sparsePCA-direct-formulation-2007} that solves approximately the dual problem of the original semidefinite programming problem, our method deals with the primal problem directly and solves it exactly, which guarantees that the resulting matrix is a sparse matrix. Global convergence result is established for the proposed method. Numerical results on both synthetic problems
and the real applications from classification of text data and senate voting data are reported to demonstrate the efficacy of our method.
\end{abstract}



\section{Introduction}\label{sec:intro} Principal component analysis (PCA) plays an important role in applications arising from data analysis,
dimension reduction and bioinformatics etc. PCA finds a few linear combinations of the original variables. These linear combinations, which are
called principal components (PC), are orthogonal to each other and explain most of the variance of the data. Specifically, for a given data matrix
$M\in\br^{p\times n}$ which consists of $n$ samples of the $p$ variables, PCA corresponds to a singular value decomposition (SVD) of $M$ or an
eigenvalue decomposition of the sample covariance matrix $\Sigma=MM^\top\in\br^{p\times p}$. Thus, for a given sample covariance matrix $\Sigma$,
PCA is usually formulated as an eigenvalue problem:
\be\label{prob:PCA} x^* := \arg\max \quad x^\top \Sigma x, \quad \st \quad \|x\|_2 \leq 1, \ee where $\|x\|_2$ is the Euclidean norm of vector $x$.
Problem \eqref{prob:PCA} gives the eigenvector that corresponds to the largest eigenvalue of $\Sigma$. However, the loading vector $x^*$ is not
expected to have many zero coefficients. This makes it hard to explain the PCs. For example, in the text classification problem, we are given a
binary data matrix $M\in\br^{p\times n}$ that records the occurrences of $p$ words in $n$ postings. That is, $M_{ij}=1$ if the $i$-th word appears
in the $j$-th posting and $M_{ij}=0$ if the $i$-th word does not appear in the $j$-th posting. The standard PCA cannot tell which words contribute most to the explained variance since the loadings are linear combinations of all the variables.
Thus, sparse PCs are needed because it is easier to analyze which variables contribute most to the explained variance.

Many techniques were proposed to extract sparse PCs from given sample covariance matrix $\Sigma$ or sample data matrix $M$. One natural thought is to
impose a cardinality constraint to \eqref{prob:PCA}, which leads to the following formulation for sparse PCA:
\be\label{prob:PCA-card-cons} x^* := \arg\max \quad x^\top \Sigma x, \quad \st \quad \|x\|_2 \leq 1, \quad \|x\|_0 \leq K, \ee
where $\|x\|_0$ (the $\ell_0$ norm of $x$) counts the number of nonzeros of $x$ and the integer $K$ controls the sparsity of the solution.
Since the cardinality constraint $\|x\|_0\leq K$ makes the problem numerically intractable, many different models were proposed in the literature to overcome
this difficulty.

In \cite{daspremont-sparsePCA-direct-formulation-2007}, d'Aspremont \etal proposed to approximately solve \eqref{prob:PCA-card-cons} by its convex relaxation,
which is a nonsmooth semidefinite programming (SDP) problem. This is the first work that attempts to approximately solve
\eqref{prob:PCA-card-cons} by a convex problem.
The SDP formulation is based on the lifting and projection technique, which is a standard technique in using SDP to approximate combinatorial problems
(see e.g., \cite{Alizadeh93interiorpoint,BV-ConvexBook2004,Todd2001}). Note that if we denote $X=xx^\top$,
then \eqref{prob:PCA-card-cons} can be rewritten as
\be\label{prob:PCA-dAspremont-sdp-con-L0} \max_{X\in\br^{p\times p}}\{\langle \Sigma, X\rangle, \st \Tr(X) = 1, \|X\|_0 \leq K^2, X\succeq 0, \rank(X) = 1\}, \ee
where $\Tr(X)$ denotes the trace of matrix $X$. The rank constraint is then dropped and the cardinality constraint is replaced by $\ell_1$ norm constraint,
and this leads to following convex problem, which is an SDP.
\be\label{prob:PCA-dAspremont-sdp-con-L1}\max_{X\in\br^{p\times p}}\{\langle \Sigma, X\rangle, \st \Tr(X) = 1, \|X\|_1 \leq K, X\succeq 0\},\ee where the $\ell_1$ norm of $X$ is defined as $\|X\|_1:=\sum_{ij}|X_{ij}|$ and using the convex constraint $\|X\|_1 \leq K$ to impose the sparsity of the solution is inspired by the recent emergence of compressed sensing (see e.g., \cite{Candes-Romberg-Tao-2006,Donoho-06}).
Note that $\|X\|_1\leq K$ is used in \eqref{prob:PCA-dAspremont-sdp-con-L1} instead of $\|X\|_1\leq K^2$. This is due to the fact that,
when $X=xx^\top$ and $\Tr(X) = 1$, we have $\|X\|_F=1$, and also that if $\|X\|_0\leq K^2$, then $\|X\|_1\leq K\|X\|_F$.
After the optimal solution $X^*$ to \eqref{prob:PCA-dAspremont-sdp-con-L1} is obtained, the vector $\hat{x}$ from the rank-one decomposition of $X^*$,
i.e., $X^*=\hat{x}\hat{x}^\top$ is used as an approximation of the solution of \eqref{prob:PCA-card-cons}.
This is the whole procedure of the lifting and projection technique.
Although some standard methods such as interior point methods can be used to solve the SDP \eqref{prob:PCA-dAspremont-sdp-con-L1}
(see e.g., \cite{Alizadeh93interiorpoint,BV-ConvexBook2004,Todd2001}), it is not wise to do so because \eqref{prob:PCA-dAspremont-sdp-con-L1}
is a nonsmooth problem, and transferring it to a standard SDP increases the size of the problem dramatically.

It is known that \eqref{prob:PCA-dAspremont-sdp-con-L1} is equivalent to the following problem with an appropriately chosen parameter $\rho > 0$:
\be\label{prob:PCA-dAspremont-sdp-pen-L1}\max_{X\in\br^{p\times p}} \{\langle \Sigma, X\rangle - \rho\|X\|_1 \st \Tr(X) = 1, X\succeq 0\}. \ee
Note that \eqref{prob:PCA-dAspremont-sdp-pen-L1} can be rewritten as
\be \label{prob:PCA-dAspremont-sdp-pen-L1-U} \max_{X\succeq 0, \Tr(X) = 1} \min_{\|U\|_\infty\leq \rho} \langle \Sigma+U,X\rangle, \ee
where $\|U\|_\infty$ denotes the largest component of $U$ in magnitude, i.e., $\|U\|_\infty =\max_{ij}|U_{ij}|$.
The dual problem of \eqref{prob:PCA-dAspremont-sdp-pen-L1} is given by interchanging the max and min in \eqref{prob:PCA-dAspremont-sdp-pen-L1-U},
i.e., \[\min_{\|U\|_\infty\leq \rho} \max_{X\succeq 0, \Tr(X) = 1}  \langle \Sigma+U,X\rangle,\] which can be further reduced to
\be \label{prob:PCA-dAspremont-sdp-pen-L1-dual} \min_{U\in\br^{p\times p}} \lambda_{\max}(\Sigma+U), \quad \st \quad \|U\|_\infty\leq \rho, \ee
where $\lambda_{\max}(Z)$ denotes the largest eigenvalue of matrix $Z$. d'Aspremont \etal \cite{daspremont-sparsePCA-direct-formulation-2007} proposed to solve the dual problem \eqref{prob:PCA-dAspremont-sdp-pen-L1-dual} using Nesterov's first-order algorithm (see e.g., \cite{Nesterov-1983,Nesterov-2005}), which is an accelerated projected gradient method. However, since the objective function of \eqref{prob:PCA-dAspremont-sdp-pen-L1-dual}
is nonsmooth, one needs to smooth it in order to apply Nesterov's algorithm. Thus, the authors of \cite{daspremont-sparsePCA-direct-formulation-2007}
actually solve an approximation of the dual problem \eqref{prob:PCA-dAspremont-sdp-pen-L1-dual}, which can be formulated as follows.
\be\label{prob:PCA-dAspremont-sdp-pen-L1-dual-NesterovSmooth} \min \quad f_\mu(U), \st \|U\|_\infty\leq\rho, \ee
where $\mu>0$ is the smoothing parameter, $f_\mu(U):=\max\{\langle \Sigma+U,X \rangle - \mu d(X), \st \Tr(X)=1,X\succeq 0\}$ and $d(X):=\Tr(X\log X)+\log(n)$.
It is shown in \cite{daspremont-sparsePCA-direct-formulation-2007} that an approximate solution $X^k$ to the primal problem \eqref{prob:PCA-dAspremont-sdp-pen-L1} can be obtained by $X^k=\nabla f_\mu(U^k)$, where $U^k$ is an approximate solution of \eqref{prob:PCA-dAspremont-sdp-pen-L1-dual-NesterovSmooth}. It is easy to see that $X^k$ is not guaranteed to be a sparse matrix.
Besides, although there is no duality gap between \eqref{prob:PCA-dAspremont-sdp-pen-L1} and \eqref{prob:PCA-dAspremont-sdp-pen-L1-dual},
the authors solve an approximation of \eqref{prob:PCA-dAspremont-sdp-pen-L1-dual}. It needs also to be noted that Nesterov's algorithm used
in \cite{daspremont-sparsePCA-direct-formulation-2007} cannot solve the constrained problem \eqref{prob:PCA-dAspremont-sdp-con-L1}.
Although \eqref{prob:PCA-dAspremont-sdp-con-L1} and \eqref{prob:PCA-dAspremont-sdp-pen-L1} are equivalent with appropriately chosen
parameters $K$ and $\rho$, in many applications, it is usually easier to choose an appropriate $K$ since we know how many nonzeros are
preferred in the sparse PCs.

Nonconvex reformulations of \eqref{prob:PCA-card-cons} include the followsings. Zou \etal \cite{Zou-spca-2006} considered a regression type formulation of \eqref{prob:PCA-card-cons} with Lasso and
elastic net regularizations. d'Aspremont \etal \cite{daspremont-sparsePCA-JMLR-2008} considered a penalty version of \eqref{prob:PCA-card-cons},
\be\label{prob:PCA-card-penalty} \phi(\rho) \equiv \max_{\|x\|_2\leq 1} x^\top \Sigma x - \rho \|x\|_0.\ee d'Aspremont \cite{daspremont-sparsePCA-JMLR-2008}
showed that \eqref{prob:PCA-card-penalty} is equivalent to the following problem that maximizes a convex function over spherical constraint:
\be\label{prob:PCA-card-penalty-equiv-max-convex} \phi(\rho) = \max_{\|x\|_2=1} \sum_{i=1}^p ((a_i^\top x)^2-\rho)_+, \ee
where $(\alpha)_+:=\max\{\alpha,0\}$, $\Sigma = A^\top A$ and $a_i$ is the $i$-th column of $A\in\br^{p\times p}$. Clearly,
\eqref{prob:PCA-card-penalty-equiv-max-convex} is a non-convex problem.  d'Aspremont \etal thus proposed in \cite{daspremont-sparsePCA-JMLR-2008}
to solve \eqref{prob:PCA-card-penalty-equiv-max-convex} by a greedy method. Journee \etal \cite{Journee-Nesterov-sparsePCA-JMLR-2010} considered
the same formulation \eqref{prob:PCA-card-penalty-equiv-max-convex} and proposed a gradient type method, which is actually a generalized power method,
to solve it.

Note that \eqref{prob:PCA-card-cons} only gives the largest sparse PC.
In many applications, several leading sparse PCs are needed in order to explain more variance. Multiple sparse PCs are usually found by solving a sequence of sparse PCA problems \eqref{prob:PCA-card-cons}, with $\Sigma$ constructed via the so-called {\it deflation} technique for each sparse PC.
Lu and Zhang \cite{Lu-Zhang-sparsePCA-MPA-2011}
proposed the following model to compute the leading $r$ sparse PCs of $\Sigma$ simultaneously:
\be\label{prob:PCA-zhaosong} \ba{ll}
\displaystyle\max_{V\in\br^{p\times r}} & \Tr(V^\top\Sigma V) - \rho \|V\|_1 \\ \st & |V_i^\top\Sigma V_j| \leq \Delta_{ij}, \forall i\neq j, \\ & V^\top V = I,
\ea\ee where each column of $V$ corresponds to a loading vector of the sample covariance matrix $\Sigma$ and $\Delta_{ij}\geq 0 (i\neq j)$ are
the parameters that control the correlation of the PCs.
Lu and Zhang \cite{Lu-Zhang-sparsePCA-MPA-2011} proposed an augmented Lagrangian method to solve \eqref{prob:PCA-zhaosong}.
Note that for these nonconvex formulations, algorithms proposed in the literature usually have only local convergence and global convergence is not guaranteed.

In this paper, we propose an alternating direction method based on a variable-splitting technique and an augmented Lagrangian framework
for solving directly the primal problems \eqref{prob:PCA-dAspremont-sdp-con-L1} and \eqref{prob:PCA-dAspremont-sdp-pen-L1}.
Our method solves two subproblems in each iteration. One subproblem has a closed-form solution that corresponds to projecting
a given matrix onto the simplex of the cone of semidefinite matrices. This projection requires an eigenvalue decomposition.
The other subproblem has a closed-form solution that corresponds to a vector shrinkage operation (for Problem \eqref{prob:PCA-dAspremont-sdp-pen-L1})
or a projection onto the $\ell_1$ ball (for Problem \eqref{prob:PCA-dAspremont-sdp-con-L1}). Thus, our method produces two iterative points at each iteration.
One iterative point is a semidefinite matrix with trace equal to one and the other one is a sparse matrix.
Eventually these two points will converge to the same point, and thus we get an optimal solution which is a sparse and semidefinite matrix.
Compared with the Nesterov's first-order method suggested in \cite{daspremont-sparsePCA-direct-formulation-2007} for
solving the approximated dual problem \eqref{prob:PCA-dAspremont-sdp-pen-L1-dual-NesterovSmooth}, our method can solve the nonsmooth primal problems
\eqref{prob:PCA-dAspremont-sdp-con-L1} and \eqref{prob:PCA-dAspremont-sdp-pen-L1} uniformly. Also, since we deal with the primal problems directly,
the $\ell_1$ norm in the constraint or the objective function guarantees that our solution is a sparse matrix, while Nesterov's method in
\cite{daspremont-sparsePCA-direct-formulation-2007} does not guarantee
this since it solves the approximated dual problem.

The rest of the paper is organized as follows. In Section \ref{sec:admm},
we introduce our alternating direction method of multipliers for solving the nonsmooth SDP problems \eqref{prob:PCA-dAspremont-sdp-con-L1}
and \eqref{prob:PCA-dAspremont-sdp-pen-L1}. The global convergence results of the alternating direction method of multipliers are given
in Section \ref{sec:convergence}. We discuss some practical issues including the deflation technique for computing multiple sparse PCs
in Section \ref{sec:deflation}. In Section \ref{sec:numerical}, we use our alternating direction method of multipliers to solve sparse
PCA problems arising from different applications such as classification of text data and senate voting records to demonstrate the efficacy of our method.
We make some conclusions in Section \ref{sec:conclusion}.

\section{Alternating Direction Method of Multipliers}\label{sec:admm}
We first introduce some notation. We use $\CCal$ to denote the simplex of the cone of the semidefinite matrices, i.e.,
$\CCal=\{X\in\br^{p\times p}\mid \Tr(X)=1,X\succeq 0\}$. We use $\BCal$ to denote the $\ell_1$-ball with radius $K$ in $\br^{p\times p}$,
i.e., $\BCal=\{X\in\br^{p\times p}\mid \|X\|_1\leq K\}$. $I_\ACal(X)$ denotes the indicator function of set $\ACal$, i.e.,
\be\label{def-indicator-function}I_\ACal(X) = \left\{\ba{ll} 0 & \mbox{if }X\in \ACal, \\ +\infty & \mbox{otherwise}.\ea\right.\ee
We know that $I_\CCal(X)$ and $I_\BCal(X)$ are both convex functions since $\CCal$ and $\BCal$ are both convex sets.
We then can reformulate \eqref{prob:PCA-dAspremont-sdp-con-L1} and \eqref{prob:PCA-dAspremont-sdp-pen-L1} uniformly as the following unconstrained problem:
\be \label{min-sum-indicator-general} \min \quad -\langle\Sigma,X \rangle + I_\CCal(X) + h(X),\ee where $h(X)=I_\BCal(X)$ for
\eqref{prob:PCA-dAspremont-sdp-con-L1} and $h(X)=\rho\|X\|_1$ for \eqref{prob:PCA-dAspremont-sdp-pen-L1}.
Note that $h(X)$ is convex in both cases. \eqref{min-sum-indicator-general} can be also viewed as the following inclusion problem:
\be \label{inclusion-operator} \mbox{Find } X, \st 0\in -\Sigma + \partial I_\CCal(X) + \partial h(X).\ee
Problem \eqref{inclusion-operator} finds zero of the sum of two monotone operators.
Methods based on operator-splitting techniques, such as Douglas-Rachford method and Peachman-Rachford method,
are usually used to solve Problem \eqref{inclusion-operator}
(see e.g., \cite{Douglas-Rachford-56,Peaceman-Rachford-55,Lions-Mercier-79,Eckstein-thesis-89,Eckstein-Bertsekas-1992,Combettes-Pesquet-DR-2007,Combettes-Wajs-05}).
From the convex optimization perspective, the alternating direction method of multipliers (ADMM) for solving \eqref{min-sum-indicator-general}
is a direct application of the Douglas-Rachford method. ADMM has been successfully used to solve structured convex optimization problems arising
from image processing, compressed sensing, machine learning, semidefinite programming etc.
(see e.g., \cite{Glowinski-LeTallec-89,Gabay-83,Wang-Yang-Yin-Zhang-2008,Yang-Zhang-2009,Goldstein-Osher-08,Yuan-2009,Scheinberg-Ma-Goldfarb-NIPS-2010,
Goldfarb-Ma-Ksplit,Goldfarb-Ma-Scheinberg-2010,
Malick-Povh-Rendl-Wiegele-2009,Wen-Goldfarb-Yin-2009,Boyd-etal-ADM-survey-2011}).
We now show how ADMM can be used to solve the sparse PCA problem \eqref{min-sum-indicator-general}.

ADMM is based on a variable-splitting technique and an augmented Lagrangian framework. By introducing a new variable $Y$,
\eqref{min-sum-indicator-general} can be rewritten as
\be \label{min-sum-indicator-general-xy} \ba{ll} \min & -\langle\Sigma,X \rangle + I_\CCal(X) + h(Y) \\ \st & X = Y. \ea\ee
Note that although the number of variables is increased, the two nonsmooth functions $I_\CCal(\cdot)$ and $h(\cdot)$ are now
separated since they are associated with different variables.
For this equality-constrained problem, augmented Lagrangian method is a standard approach to solve it.
A typical iteration of augmented Lagrangian method for solving \eqref{min-sum-indicator-general-xy} is given by:
\be \label{alg:AL-min-sum-indicator}
\left\{\ba{lll}(X^{k+1},Y^{k+1}) & := & \displaystyle\arg\min_{(X,Y)} \LCal_\mu(X,Y;\Lambda^k) \\
\Lambda^{k+1} & := & \Lambda^k - \frac{1}{\mu}(X^{k+1}-Y^{k+1}), \ea\right.\ee
where the augmented Lagrangian function $\LCal_\mu(X,Y;\Lambda)$ is defined as:
\be \label{augmented-lagrangian-function}
\LCal_\mu(X,Y;\Lambda) := -\langle \Sigma,X\rangle + I_\CCal(X) + h(Y) - \langle \Lambda, X-Y\rangle + \frac{1}{2\mu}\|X-Y\|_F^2, \ee
where $\mu>0$ is a penalty parameter and $\Lambda$ is the Lagrange multiplier associated with the linear constraint $X=Y$.
Note that it is usually hard to minimize the augmented Lagrangian function $\LCal_\mu(X,Y;\Lambda^k)$ with respect to $X$
and $Y$ simultaneously. In fact, it is as difficult as solving the original problem \eqref{min-sum-indicator-general-xy}.
However, if we minimize the augmented Lagrangian function with respect to $X$ and $Y$ alternatingly, we obtain two subproblems
in each iteration and both of them are relatively easy to solve. This results in the following alternating direction method of multipliers.
\be\label{alg:ADMM-min-sum-indicator}\left\{\ba{lll} \displaystyle X^{k+1} & := & \arg\min_{X}\LCal_\mu(X,Y^k;\Lambda^k) \\
                                        \displaystyle  Y^{k+1} & := & \arg\min_Y \LCal_\mu(X^{k+1},Y;\Lambda^k) \\
                                         \Lambda^{k+1} & := & \Lambda^k - (X^{k+1}-Y^{k+1})/\mu, \ea\right.\ee
It can be shown that the two subproblems in \eqref{alg:ADMM-min-sum-indicator} are both relatively easy to solve in the
sparse PCA problem. Before we do that, we characterize two nice properties of the indicator function \eqref{def-indicator-function}.

\begin{itemize}
\item {\bf Property 1.} The proximal mapping of the indicator function $I_\ACal(\cdot)$ is the Euclidean projection onto $\ACal$, i.e.,
\be \label{prox-equal-projection} \prox_{I_\ACal}(X) \equiv \PCal_{\ACal}(X),\ee where
\be \label{def-proximal-mapping} \prox_{I_\ACal}(X) := \arg\min_U \{I_\ACal(U)+\half\|U-X\|_F^2\},\ee
and
\be \label{def-projection} \PCal_{\ACal}(X) := \arg\min_U \{\half\|U-X\|_F^2, \st U\in\ACal\}.\ee
\item {\bf Property 2.} The optimality conditions for Problem \eqref{def-projection} are given by
\be \label{opt-cond-projection-subgradient} X-U^* \in \partial I_\ACal(U^*),\ee which is equivalent to
\be \label{opt-cond-projection-subgradient-equiv} \langle X-U^*,Z-U^* \rangle \leq 0, \forall Z\in\ACal, \ee
where $U^*$ is the optimal solution of \eqref{def-projection}.
\end{itemize}

Now, the first subproblem in \eqref{alg:ADMM-min-sum-indicator} can be reduced to:
\be\label{alg:ADAL-min-sum-indicator-X-ind} X^{k+1} := \arg\min \left\{ \mu I_\CCal(X) + \half \|X-(Y^k+\mu\Lambda^k+\mu\Sigma)\|_F^2\right\}, \ee
which can be further reduced to projection onto $\CCal$ using Property 1,
\be \label{alg:ADAL-min-sum-indicator-X-proj}
X^{k+1} = \PCal_\CCal(Y^k+\mu\Lambda^k+\mu\Sigma) := \arg\min \left\{ \half \|X-(Y^k+\mu\Lambda^k+\mu\Sigma)\|_F^2, \st X\in\CCal\right\}. \ee

When $h(Y)=I_\BCal(Y)$ as in Problem \eqref{prob:PCA-dAspremont-sdp-con-L1}, the second subproblem in \eqref{alg:ADMM-min-sum-indicator} can be reduced to:
\be \label{alg:ADAL-min-L1-sub-Y-ind} Y^{k+1} := \arg\min \left\{\mu I_\BCal(Y) + \half\|Y-(X^{k+1}-\mu\Lambda^k)\|_F^2 \right\},\ee
which can be further reduced to projection onto $\BCal$ using Property 1,
\be \label{alg:ADAL-min-L1-sub-Y-proj}
Y^{k+1} = \PCal_\BCal(X^{k+1}-\mu\Lambda^k) := \arg\min \left\{\half\|Y-(X^{k+1}-\mu\Lambda^k)\|_F^2, \st Y\in\BCal \right\}.\ee

When $h(Y)=\rho\|Y\|_1$ as in Problem \eqref{prob:PCA-dAspremont-sdp-pen-L1}, the second subproblem in \eqref{alg:ADMM-min-sum-indicator} can be reduced to:
\be \label{alg:ADAL-min-L1-sub-Y-penL1} Y^{k+1} := \arg\min \left\{\mu\rho\|Y\|_1 + \half\|Y-(X^{k+1}-\mu\Lambda^k)\|_F^2 \right\}.\ee
Problem \eqref{alg:ADAL-min-L1-sub-Y-penL1} has a closed-form solution that is given by
\be \label{alg:ADAL-min-L1-sub-Y-penL1-sol} Y^{k+1} = \Shrink(X^{k+1}-\mu\Lambda^k,\mu\rho), \ee where the shrinkage operator is defined as:
\be \label{shrinkage} (\Shrink(Z,\tau))_{ij} := \sgn(Z_{ij})\max\{|Z_{ij}|-\tau,0\}, \forall i,j. \ee

In the following, we will show that \eqref{alg:ADAL-min-sum-indicator-X-ind} and \eqref{alg:ADAL-min-L1-sub-Y-ind} are easy to solve,
i.e., the two projections \eqref{alg:ADAL-min-sum-indicator-X-proj} and \eqref{alg:ADAL-min-L1-sub-Y-proj} can be done efficiently.
First, since the problem of projection onto $\CCal$
\be \label{projection-simlex-mat} \PCal_\CCal(X) = \arg\min \{ \half\|Z-X\|_F^2, \st \Tr(Z) =1, Z\succeq 0 \}\ee
is unitary-invariant, its solution is given by
$\PCal_\CCal(X)=U\diag(\gamma)U^\top$, where $X=U\diag(\sigma)U^\top$ is the eigenvalue decomposition of $X$, and $\gamma$
is the projection of $\sigma$ onto the simplex in the Euclidean space, i.e.,
\be \label{projection-simplex-vec} \gamma := \arg\min \{\half\|\xi-\sigma\|_2^2, \st \sum_{i=1}^p \xi_i = 1, \xi\geq 0\}.\ee

We consider a slightly more general problem
\be \label{projection-simplex-vec-general} \xi^* := \arg\min \{\half\|\xi-\sigma\|_2^2, \st \sum_{i=1}^p \xi_i = r, \xi\geq 0\},\ee
where scalar $r >0$. Note that \eqref{projection-simplex-vec} is a special case of \eqref{projection-simplex-vec-general} with $r=1$.
From the first-order optimality conditions for \eqref{projection-simplex-vec-general}, it is easy to show that the optimal solution of
\eqref{projection-simplex-vec-general} is given by
\[\xi^*_i:=\max\{\sigma_i-\theta,0\}, \forall i=1,\ldots,p,\] where the scalar $\theta$ is the solution of the following piecewise linear equation:
\be\label{projection-simplex-vec-general-sol-linear-equation}\sum_{i=1}^p\max\{\sigma_i-\theta,0\}=r.\ee It is known that the piecewise linear equation
\eqref{projection-simplex-vec-general-sol-linear-equation} can be solved quite efficiently and thus solving \eqref{projection-simplex-vec-general} can be done easily.
In fact, the following procedure (Algorithm \ref{alg:projection-simplex-vector}) gives the optimal solution of \eqref{projection-simplex-vec-general}.
We refer the readers to \cite{Shalev-Shwartz-Singer-JMLR-2006} for the proof of the validity of the algorithm.
\begin{algorithm2e}\caption{Projection onto the simplex in the Euclidean space} \label{alg:projection-simplex-vector}
\linesnumberedhidden \dontprintsemicolon
Input: A vector $\sigma\in\br^p$ and a scalar $r>0$. \;
Sort $\sigma$ into $\hat{\sigma}$ as a non-decreasing order: $\hat\sigma_1\leq\hat\sigma_2\leq\ldots\leq\hat\sigma_p$ \;
Find index $\hat{j}$, the smallest $j$ such that $\hat\sigma_j - \frac{1}{p-j+1}\left(\displaystyle\sum_{i=j}^p \hat\sigma_i - r\right) > 0$ \;
Compute $\theta = \frac{1}{p-\hat{j}+1}\left(\displaystyle\sum_{i=\hat{j}}^p \hat\sigma_i - r\right)$ \;
Output: A vector $\gamma$, s.t. $\gamma_i = \max\{\sigma_i-\theta,0\},i=1,\ldots,p.$ \;
\end{algorithm2e}
It is easy to see that Algorithm \ref{alg:projection-simplex-vector} has an $O(p\log p)$ complexity. Linear time algorithms for solving
\eqref{projection-simplex-vec-general} are studied in
\cite{Brucker-1984,Pardalos-Kovoor-1990,Duchi-projection-L1-ball-ICML-2008}.
Thus, solving \eqref{alg:ADAL-min-sum-indicator-X-ind} corresponds to an eigenvalue decomposition and a projection onto the simplex in the Euclidean space, and
they both can be done efficiently.

Solving \eqref{alg:ADAL-min-L1-sub-Y-ind} (or equivalently \eqref{alg:ADAL-min-L1-sub-Y-proj}) corresponds to a projection onto the $\ell_1$-ball: $\|Y\|_1\leq K$.
It has been shown in \cite{Duchi-projection-L1-ball-ICML-2008,vandenBerg-Friedlander-2008} that projection onto the $\ell_1$-ball can be done easily.
In fact, the solution of
\be \label{projection-L1-ball} \hat\gamma = \arg\min \{\half\|\xi-\hat\sigma\|_2^2, \st \|\xi\|_1\leq r\} \ee
is given by $\hat\gamma_i=\sgn(\hat\sigma_i)\gamma_i, \forall i=1,\ldots,p$, where $\gamma$ is the solution of
\[\min \quad \half\|\gamma-|\hat\sigma|\|_2^2, \quad \st \sum_{i=1}^p \gamma_i=r, \gamma\geq 0,\]
i.e., the projection of $|\hat\sigma|$ (elementwise absolute value of $\hat\sigma$) onto the simplex.
Thus, \eqref{alg:ADAL-min-L1-sub-Y-ind} can be rewritten as
\be \label{alg:ADAL-min-L1-sub-Y-ind-sol} \Xvec(Y^{k+1}) = \arg\min \{\half\|y-\Xvec(X^{k+1}-\mu\Lambda^k)\|_2^2, \st \|y\|_1\leq K\}, \ee
and it corresponds to a projection onto the simplex in the Euclidean space, where $\Xvec(Y)$ denotes the vector form of $Y$ which is obtained by stacking the columns of $Y$ into a long vector.

To summarize, our ADMM for solving \eqref{prob:PCA-dAspremont-sdp-con-L1} and \eqref{prob:PCA-dAspremont-sdp-pen-L1} can be uniformly described
as Algorithm \ref{alg:ADAL-min-L1-final}.
\begin{algorithm2e}\caption{ADMM for solving \eqref{prob:PCA-dAspremont-sdp-con-L1} and \eqref{prob:PCA-dAspremont-sdp-pen-L1}} \label{alg:ADAL-min-L1-final}
\linesnumberedhidden  \dontprintsemicolon
Initialization: $Y^0=0$, $\Lambda^0=0$. \;
\For {k=0,1,\ldots}{
Compute the eigenvalue decomposition: $Y^k+\mu\Lambda^k+\mu\Sigma=U\diag(\sigma)U^\top$ \;
Project $\sigma$ onto the simplex in Euclidean space by Algorithm \ref{alg:projection-simplex-vector}, and denote the solution by $\gamma$ \;
Compute $X^{k+1} = U\diag(\gamma)U^\top$ \;
Perform one of the followings:
\begin{itemize}
\item if \eqref{prob:PCA-dAspremont-sdp-con-L1} is solved, update $Y^{k+1}$ by solving \eqref{alg:ADAL-min-L1-sub-Y-ind-sol} \;
\item if \eqref{prob:PCA-dAspremont-sdp-pen-L1} is solved, update $Y^{k+1}$ by \eqref{alg:ADAL-min-L1-sub-Y-penL1-sol} \;
\end{itemize}
Update $\Lambda^{k+1}$ by $\Lambda^{k+1} = \Lambda^k - (X^{k+1}-Y^{k+1})/\mu$ \;
}
\end{algorithm2e}

\begin{remark}
Although Algorithm \ref{alg:ADAL-min-L1-final} suggests that we need to compute the eigenvalue decomposition of $Y^k+\mu\Lambda^k+\mu\Sigma$ in order to
get the solution to \eqref{alg:ADAL-min-sum-indicator-X-ind}, we actually only need to compute the positive eigenvalues and corresponding eigenvectors
of $Y^k+\mu\Lambda^k+\mu\Sigma$.
\end{remark}

\section{Global Convergence Results}\label{sec:convergence}

In this section, we prove that the sequence $(X^k,Y^k,\Lambda^k)$ produced by the alternating direction method of multipliers \eqref{alg:ADMM-min-sum-indicator}
(i.e., Algorithm \ref{alg:ADAL-min-L1-final}) converges to $(X^*,Y^*,\Lambda^*)$, where $(X^*,Y^*)$ is an optimal solution to \eqref{min-sum-indicator-general-xy}
and $\Lambda^*$ is the corresponding optimal dual variable. Although the proof of global convergence results of
ADMM has been studied extensively in the literature (see e.g., \cite{Eckstein-Bertsekas-1992,He-Liao-Han-Yang-2002}),
we here give a very simple proof of the convergence of our ADMM that utilizes the special structures of the sparse PCA problem.
We only prove the case when $h(Y)=I_\BCal(Y)$ and leave the case when $h(Y)=\rho\|Y\|_1$ to the readers since their proofs are almost identical.

Before we give the main theorem about the global convergence of \eqref{alg:ADMM-min-sum-indicator} (Algorithm \ref{alg:ADAL-min-L1-final}),
we need the following lemma.
\begin{lemma}\label{lem:Fejer-monotone}
Assume that $(X^*,Y^*)$ is an optimal solution of \eqref{min-sum-indicator-general-xy} and $\Lambda^*$ is the corresponding optimal dual
variable associated with the equality constraint $X=Y$. Then the sequence $(X^k,Y^k,\Lambda^k)$ produced by \eqref{alg:ADMM-min-sum-indicator} satisfies
\be \label{lem:conclusion-eq} \|U^{k}-U^*\|_G^2 - \|U^{k+1}-U^*\|_G^2 \geq \|U^k - U^{k+1}\|_G^2, \ee
where $U^* = \begin{pmatrix}\Lambda^* \\ Y^* \end{pmatrix}$, $U^k = \begin{pmatrix}\Lambda^k \\ Y^k \end{pmatrix}$
and $G = \begin{pmatrix} \mu I & 0 \\ 0 & \frac{1}{\mu}I \end{pmatrix}$, and the norm $\|\cdot\|_G^2$ is defined
as $\|U\|_G^2 = \langle U, GU\rangle$ and the corresponding inner product $\langle \cdot,\cdot\rangle_G$ is defined
as $\langle U,V\rangle_G = \langle U,GV\rangle$.
\end{lemma}
\begin{proof}
Since $(X^*,Y^*,\Lambda^*)$ is optimal to \eqref{min-sum-indicator-general-xy}, it follows from the KKT conditions that the followings hold:
\be \label{lem-KKT-X} 0 \in -\Sigma + \partial I_\CCal(X^*) - \Lambda^*, \ee
\be \label{lem-KKT-Y} 0 \in \partial I_\BCal(Y^*) + \Lambda^*, \ee
and
\be \label{lem-KKT-X=Y} X^* = Y^*\in\CCal \cap\BCal.\ee
By using Property 2, \eqref{lem-KKT-X} and \eqref{lem-KKT-Y} can be respectively reduced to:
\be \label{lem-KKT-X-reduced} \langle \Sigma+\Lambda^*,X-X^*\rangle \leq 0, \forall X\in\CCal, \ee
and
\be \label{lem-KKT-Y-reduced} \langle -\Lambda^*,Y-Y^*\rangle \leq 0, \forall Y\in\BCal. \ee

Note that the optimality conditions for the first subproblem (i.e., the subproblem with respect to $X$) in \eqref{alg:ADMM-min-sum-indicator}
are given by $X^{k+1}\in\CCal$ and
\be \label{lem-OPTcond-X} 0\in-\Sigma+\partial I_\CCal(X^{k+1}) - \Lambda^k + \frac{1}{\mu}(X^{k+1}-Y^k). \ee
By using Property 2 and the updating formula for $\Lambda^k$ in \eqref{alg:ADMM-min-sum-indicator}, i.e.,
\be \label{update-Lambda} \Lambda^{k+1} = \Lambda^k - \frac{1}{\mu}(X^{k+1}-Y^{k+1}),\ee
\eqref{lem-OPTcond-X} can be rewritten as
\be \label{lem-OPTcond-X-reduced} \langle \Sigma+\Lambda^{k+1}+\frac{1}{\mu}(Y^k-Y^{k+1}), X-X^{k+1} \rangle \leq 0, \forall X\in\CCal. \ee
Letting $X=X^{k+1}$ in \eqref{lem-KKT-X-reduced} and $X=X^*$ in \eqref{lem-OPTcond-X-reduced}, and summing the two resulting inequalities, we get,
\be \label{lem-OPTcond-X-final} \langle \Lambda^{k+1}-\Lambda^*+\frac{1}{\mu}(Y^k-Y^{k+1}), X^*-X^{k+1}\rangle \leq 0. \ee

The optimality conditions for the second subproblem (i.e., the subproblem with respect to $Y$) in \eqref{alg:ADMM-min-sum-indicator}
are given by $Y^{k+1}\in\BCal$ and
\be \label{lem-OPTcond-Y} 0\in\partial I_\BCal(Y^{k+1}) + \Lambda^k + \frac{1}{\mu}(Y^{k+1}-X^{k+1}). \ee
By using Property 2 and \eqref{update-Lambda},
\eqref{lem-OPTcond-Y} can be rewritten as
\be \label{lem-OPTcond-Y-reduced} \langle -\Lambda^{k+1}, Y-Y^{k+1}\rangle \leq 0, \forall Y\in\BCal. \ee
Letting $Y=Y^{k+1}$ in \eqref{lem-KKT-Y-reduced} and $Y=Y^*$ in \eqref{lem-OPTcond-Y-reduced}, and summing the two resulting inequalities, we obtain,
\be \label{lem-OPTcond-Y-final} \langle \Lambda^*-\Lambda^{k+1}, Y^*-Y^{k+1} \rangle \leq 0. \ee

Summing \eqref{lem-OPTcond-X-final} and \eqref{lem-OPTcond-Y-final}, and using the facts that $X^*=Y^*$
and $X^{k+1}=\mu(\Lambda^k-\Lambda^{k+1})+Y^{k+1}$, we obtain,
\be \label{lem-proof-inequa-1}
\mu\langle \Lambda^k-\Lambda^{k+1},\Lambda^{k+1}-\Lambda^*\rangle + \frac{1}{\mu}\langle Y^k-Y^{k+1},Y^{k+1}-Y^*\rangle
\geq -\langle Y^k-Y^{k+1},\Lambda^k-\Lambda^{k+1}\rangle. \ee
Rearranging the left hand side of \eqref{lem-proof-inequa-1}
by using $\Lambda^{k+1}-\Lambda^*=(\Lambda^{k+1}-\Lambda^k)+(\Lambda^k-\Lambda^*)$ and $Y^{k+1}-Y^* = (Y^{k+1}-Y^k) + (Y^k-Y^*)$, we get
\be \label{lem-proof-inequa-2}
\mu\langle \Lambda^k-\Lambda^*,\Lambda^k-\Lambda^{k+1}\rangle + \frac{1}{\mu}\langle Y^k-Y^*,Y^k-Y^{k+1}\rangle
\geq \mu\|\Lambda^k-\Lambda^{k+1}\|^2 + \frac{1}{\mu}\|Y^k-Y^{k+1}\|^2 -\langle \Lambda^{k+1}-\Lambda^k, Y^{k+1}-Y^k\rangle. \ee
Using the notation of $U^k$, $U^*$ and $G$, \eqref{lem-proof-inequa-2} can be rewritten as
\be\label{lem-proof-inequa-3}
\langle U^k-U^*,U^k - U^{k+1} \rangle_G \geq \|U^k-U^{k+1}\|_G^2-\langle\Lambda^k-\Lambda^{k+1},Y^k-Y^{k+1}\rangle.\ee
Combining \eqref{lem-proof-inequa-3}
with the identity \[\|U^{k+1}-U^*\|_G^2=\|U^{k+1}-U^k\|_G^2-2\langle U^{k}-U^{k+1}, U^k-U^*\rangle_G+\|U^k-U^*\|_G^2,\]
we get \be \label{lem-proof-inequa-4}
\begin{array}{ll} & \|U^k-U^*\|_G^2 - \|U^{k+1}-U^*\|_G^2 \\ = & 2\langle U^k-U^{k+1},U^k-U^*\rangle-\|U^{k+1}-U^k\|_G^2 \\
\geq & 2\|U^k-U^{k+1}\|_G^2-2\langle\Lambda^k-\Lambda^{k+1},Y^k-Y^{k+1}\rangle-\|U^{k+1}-U^k\|_G^2 \\
= & \|U^k-U^{k+1}\|_G^2-2\langle\Lambda^k-\Lambda^{k+1},Y^k-Y^{k+1}\rangle.\end{array} \ee

Now, using \eqref{lem-OPTcond-Y-reduced} for $k$ instead of $k+1$ and letting $Y=Y^{k+1}$, we get,
\be \label{lem-proof-inequa-5} \langle -\Lambda^k, Y^{k+1}-Y^k\rangle \leq 0.\ee
Letting $Y=Y^{k}$ in \eqref{lem-OPTcond-Y-reduced} and adding it to \eqref{lem-proof-inequa-5} yields,
\be \label{lem-proof-inequa-6} \langle \Lambda^k-\Lambda^{k+1},Y^k-Y^{k+1}\rangle \leq 0.\ee
By substituting \eqref{lem-proof-inequa-6} into \eqref{lem-proof-inequa-4} we get the desired result \eqref{lem:conclusion-eq}.
\end{proof}

We are now ready to give the main convergence result of \eqref{alg:ADMM-min-sum-indicator} (Algorithm \ref{alg:ADAL-min-L1-final}).
\begin{theorem}\label{the:main-convergence}
The sequence $\{(X^k,Y^k,\Lambda^k)\}$ produced by \eqref{alg:ADMM-min-sum-indicator} (Algorithm \ref{alg:ADAL-min-L1-final})
from any starting point converges to an optimal solution to Problem \eqref{min-sum-indicator-general-xy}.
\end{theorem}
\begin{proof}
From Lemma \ref{lem:Fejer-monotone} we can easily get that
\begin{itemize}
\item (i) $\|U^k-U^{k+1}\|_G \rightarrow 0$;
\item (ii) $\{U^k\}$ lies in a compact region;
\item (iii) $\|U^k-U^*\|_G^2$ is monotonically non-increasing and thus converges.
\end{itemize}
It follows from (i) that $\Lambda^k-\Lambda^{k+1}\rightarrow 0$ and $Y^k-Y^{k+1}\rightarrow 0$.
Then \eqref{update-Lambda} implies that $X^k-X^{k+1}\rightarrow 0$ and $X^k-Y^k\rightarrow 0$.
From (ii) we obtain that, $U^k$ has a subsequence $\{U^{k_j}\}$ that converges to $\hat{U}=(\hat\Lambda,\hat{Y})$,
i.e., $\Lambda^{k_j}\rightarrow\hat\Lambda$ and $Y^{k_j}\rightarrow\hat{Y}$. From $X^k-Y^k\rightarrow 0$ we also get
that $X^{k_j}\rightarrow\hat{X}:=\hat{Y}$. Therefore, $(\hat{X},\hat{Y},\hat{\Lambda})$ is a limit point of $\{(X^k,Y^k,\Lambda^k)\}$.

Note that by using \eqref{update-Lambda}, \eqref{lem-OPTcond-X} can be rewritten as
\be \label{the-convergence-inequa-1} 0 \in -\Sigma+\partial I_\CCal(X^{k+1}) - \Lambda^{k+1} + \frac{1}{\mu}(Y^{k+1}-Y^k),\ee which implies that
\be \label{the-convergence-inequa-2} 0 \in -\Sigma+\partial I_\CCal(\hat{X}) - \hat\Lambda.\ee
Note also that \eqref{lem-OPTcond-Y} implies that
\be \label{the-convergence-inequa-3} 0 \in \partial I_\BCal(\hat{Y}) + \hat\Lambda.\ee
Moreover, it follows from $X^k\in\CCal$ and $Y^k\in\BCal$ that
\be \label{the-convergence-inequa-4} \hat{X}\in\CCal \mbox{ and } \hat{Y}\in\BCal. \ee
\eqref{the-convergence-inequa-2}, \eqref{the-convergence-inequa-3}, \eqref{the-convergence-inequa-4} together with $\hat{X}=\hat{Y}$
imply that $(\hat{X},\hat{Y},\hat{\Lambda})$ satisfies the KKT conditions for \eqref{min-sum-indicator-general-xy} and thus is an optimal
solution to \eqref{min-sum-indicator-general-xy}.
Therefore, we showed that any limit point of $\{(X^k,Y^k,\Lambda^k)\}$ is an optimal solution to \eqref{min-sum-indicator-general-xy}.
\end{proof}

\section{The Deflation Techniques and Other Practical Issues}\label{sec:deflation}

It should be noticed that the solution of Problem \eqref{prob:PCA} only gives the largest eigenvector
(the eigenvector corresponding to the largest eigenvalue) of $\Sigma$. In many applications,
the largest eigenvector is not enough to explain the total variance of the data.
Thus one usually needs to compute several leading eigenvectors to explain more variance of the data.
Hotelling's deflation method \cite{Saad-deflation-1998} is usually used to extract the leading eigenvectors sequentially.
The Hotelling's deflation method extracts the $r$-th leading eigenvector of $\Sigma$ by solving
\[x_r = \arg\max \{x^\top \Sigma_{r-1} x, \st \|x\|_2\leq 1\},\] where $\Sigma_0:=\Sigma$ and
\[\Sigma_r = \Sigma_{r-1} - x_rx_r^\top \Sigma_{r-1}x_rx_r^\top. \]
It is easy to verify that Hotelling's deflation method preserves the positive-semidefiniteness of matrix $\Sigma_r$.
However, as pointed out in \cite{Mackey-NIPS-2008}, it does not preserve the positive-semidefiniteness of $\Sigma_r$
when it comes to the sparse PCA problem \eqref{prob:PCA-card-cons}, because the solution $x_r$ is no longer
an eigenvector of $\Sigma_{r-1}$. Thus, the second leading eigenvector produced by solving the sparse PCA
problem may not explain well the variance of the data. We should point out that the deflation method used in
\cite{daspremont-sparsePCA-direct-formulation-2007} is the Hotelling's deflation method.

Several deflation techniques to overcome this difficulty for sparse PCA were proposed by Mackey in
\cite{Mackey-NIPS-2008}. In our numerical experiments, we chose to use the Schur complement deflation method in \cite{Mackey-NIPS-2008}.
The Schur complement deflation method updates matrix $\Sigma_r$ by
\be \label{deflation-schur} \Sigma_r = \Sigma_{r-1} - \frac{\Sigma_{r-1}x_r x_r^\top \Sigma_{r-1}}{x_r^\top \Sigma_{r-1} x_r}. \ee
The Schur complement deflation method has the following properties as shown in \cite{Mackey-NIPS-2008}.
(i) Schur complement deflation preserves the positive-semidefiniteness of $\Sigma_r$, i.e., $\Sigma_r\succeq 0$.
(ii) Schur complement deflation renders $x_s$ orthogonal to $\Sigma_r$ for $s\leq r$, i.e., $\Sigma_rx_s=0, \forall s\leq r$.

When we want to find the leading $r$ sparse PCs of $\Sigma$, we use ADMM to solve sequentially $r$
problems \eqref{prob:PCA-dAspremont-sdp-con-L1} or \eqref{prob:PCA-dAspremont-sdp-pen-L1}
with $\Sigma$ updated by the Schur complement deflation method \eqref{deflation-schur}. We denote the leading $r$ sparse
PCs obtained by our ADMM as $X_r = (x_1,\ldots,x_r)$. Usually the total variance explained by $X_r$ is given by $\Tr(X_r^\top \Sigma X_r)$.
However, because we do not require $x_1,\ldots,x_r$ to be orthogonal to each other when we sequentially solve the
SDPs \eqref{prob:PCA-dAspremont-sdp-con-L1} or \eqref{prob:PCA-dAspremont-sdp-pen-L1}, these loadings are correlated.
Thus, $\Tr(X_r^\top \Sigma X_r)$ will overestimate the total explained variance by $x_1,\ldots,x_r$. To alleviate the overestimated variance,
Zou \etal \cite{Zou-spca-2006} suggested that the explained total variance should be computed using the following procedure,
which was called {\it adjusted variance}:
\[AdjVar(X_r) := \Tr(R^2),\] where $X_r=QR$ is the QR decomposition of $X_r$. In our numerical experiments, we always report the adjusted variance as the explained
variance.

It is also worth noticing that the problems we solve are convex relaxations of the original problems
\eqref{prob:PCA-card-cons} and \eqref{prob:PCA-card-penalty}. Hence, one needs to postprocess the matrix $X$
obtained by solving \eqref{prob:PCA-dAspremont-sdp-con-L1} or \eqref{prob:PCA-dAspremont-sdp-pen-L1} to get
the solution to \eqref{prob:PCA-card-cons} or \eqref{prob:PCA-card-penalty}. To get the solution to the original
sparse PCA problem \eqref{prob:PCA-card-cons} or \eqref{prob:PCA-card-penalty} from the solution $X$ of the convex SDP problem, we simply perform
a rank-one decomposition to $X$, i.e., $X=xx^\top$. Since $X$ is a sparse matrix, $x$ should be a sparse vector.
This postprocessing technique is also used in \cite{daspremont-sparsePCA-direct-formulation-2007}.

Since the sequences $\{X^k\}$ and $\{Y^k\}$ generated by ADMM converge to the same point eventually, we terminate
ADMM when the difference between $X^k$ and $Y^k$ is sufficiently small. In our numerical experiments, we terminate ADMM when
\[\frac{\|X^k-Y^k\|_F}{\max\{1,\|X^k\|_F,\|Y^k\|_F\}} < 10^{-4}.\]

\section{Numerical Results}\label{sec:numerical}

In this section, we use our ADMM to solve the SDP formulations \eqref{prob:PCA-dAspremont-sdp-con-L1}
and \eqref{prob:PCA-dAspremont-sdp-pen-L1} of sparse PCA on both synthetic and real data sets. We compare the performance of ADMM with two methods for solving
sparse PCA. One method is DSPCA \cite{daspremont-sparsePCA-direct-formulation-2007} for solving \eqref{prob:PCA-dAspremont-sdp-pen-L1} and the other method
is ALSPCA \cite{Lu-Zhang-sparsePCA-MPA-2011} for solving \eqref{prob:PCA-zhaosong}. The Matlab codes of DSPCA and ALSPCA were downloaded from the authors' websites.
Note that the main parts of the DSPCA codes were actually written in C-Mex files.
Our codes were written in Matlab. All experiments were run in MATLAB 7.3.0 on a laptop with 1.66GHZ CPU and 2GB of RAM.

\subsection{A synthetic example}

We tested our ADMM on a synthetic data set suggested by Zou \etal in \cite{Zou-spca-2006}. This synthetic example has three hidden factors:
\[V_1 \sim \mathcal{N}(0,290), V_2\sim \mathcal{N}(0,300), V_3=-0.3V_1+0.925V_2+\epsilon,\]
where $\epsilon\sim \mathcal{N}(0,1)$, and $V_1$, $V_2$ and $\epsilon$ are independent. The 10 observable variables are given by the following procedure:
\[\ba{lll} X_i = V_1 + \epsilon^1_i, & \epsilon^1_i \sim \mathcal{N}(0,1), & i = 1, 2, 3, 4, \\
           X_i = V_2 + \epsilon^2_i, & \epsilon^2_i \sim \mathcal{N}(0,1), & i = 5, 6, 7, 8, \\
           X_i = V_3 + \epsilon^3_i, & \epsilon^3_i \sim \mathcal{N}(0,1), & i = 9, 10, \ea\]
where $\epsilon^j_i$ are independent for $j=1,2,3$ and $i=1,\ldots,10$. The exact covariance matrix of $(X_1,\ldots,X_{10})$ is
used to find the standard PCs by standard PCA and sparse PCs by ADMM. Note that the variances of $V_1$, $V_2$ and $V_3$ indicate
that $V_2$ is slightly more important than $V_1$ and they are both more important than $V_3$. Also, we note that the first two PCs
explain more than $99\%$ of the total variance. Thus, using the first two PCs should be able to explain most of the variance,
and the first sparse PC explains most of the variance of $V_2$ using $(X_5,X_6,X_7,X_8)$ and the second sparse PC explains most
of the variance of $V_1$ using $(X_1,X_2,X_3,X_4)$. Based on these observations, we set $K=4$ in \eqref{prob:PCA-dAspremont-sdp-con-L1}
for computing both the first and the second sparse PCs. When we computed the second sparse PC, we used the Schur complement deflation method
described in Section \ref{sec:deflation} to construct the corresponding sample covariance matrix.
The penalty parameter $\mu$ in ADMM was set to $0.8$. The PCs given by the standard PCA and
sparse PCA using our ADMM for solving \eqref{prob:PCA-dAspremont-sdp-pen-L1} and the explained variances are shown in Table \ref{tab:synthetic-Zou}.
From Table \ref{tab:synthetic-Zou} we see that ADMM gives sparse loadings using $(X_5,X_6,X_7,X_8)$ in the first PC and $(X_1,X_2,X_3,X_4)$ in the second PC.
The first two sparse PCs explain $80.41\%$ of the total variance.

\begin{table}[ht]{
\begin{center}\caption{Loadings and explained variance for the first two PCs}\label{tab:synthetic-Zou}
\begin{tabular}{|c|cc|cc|}\hline
& \multicolumn{2}{|c|}{Standard PCA} & \multicolumn{2}{|c|}{ADMM}  \\\hline

Variables & PC1 & PC2 & PC1 & PC2 \\\hline

$X_1$    &   -0.1157  &  -0.4785  &       0  &  0.5000   \\\hline
$X_2$    &   -0.1157  &  -0.4785  &       0  &  0.5000   \\\hline
$X_3$    &   -0.1157  &  -0.4785  &       0  &  0.5000   \\\hline
$X_4$    &   -0.1157  &  -0.4785  &       0  &  0.5000   \\\hline
$X_5$    &    0.3953  &  -0.1449  &  0.5000  &       0   \\\hline
$X_6$    &    0.3953  &  -0.1449  &  0.5000  &       0   \\\hline
$X_7$    &    0.3953  &  -0.1449  &  0.5000  &       0   \\\hline
$X_8$    &    0.3953  &  -0.1449  &  0.5000  &       0   \\\hline
$X_9$    &    0.4008  &   0.0095  &       0  &       0   \\\hline
$X_{10}$ &    0.4008  &   0.0095  &       0  &       0   \\\hline
Total explained variance & \multicolumn{2}{|c|}{$99.68\%$} & \multicolumn{2}{|c|}{$80.41\%$} \\\hline
\end{tabular}
\end{center}}
\end{table}

\subsection{Pit props data}
The pit props data set has been a standard benchmark for testing sparse PCA algorithms since it was introduced by Jeffers in \cite{Jeffers-1967}.
The pit props data set has 180 observations and 13 measured variables. Thus the covariance matrix $\Sigma$ is a $13\times 13$ matrix.
We used our ADMM to compute the first six sparse PCs sequentially via the Schur complement deflation technique discussed in Section \ref{sec:deflation}.
We set $K = (6,2,2,1,1,1)$ for the six problems \eqref{prob:PCA-dAspremont-sdp-con-L1} as suggested in \cite{daspremont-sparsePCA-direct-formulation-2007}.
We set $\mu=0.8$ in ADMM. The first six sparse PCs obtained by ADMM are shown in Table \ref{tab:pitprops-ADMM}.
We compared the results with ALSPCA for solving \eqref{prob:PCA-zhaosong}.
For ALSPCA, we used the parameters as suggested by the authors, i.e., $r= 6, \rho= 0.70, \Delta_{ij}= 0.50,\forall i\neq j$. The results given by ALSPCA are reported in Table \ref{tab:pitprops-ALSPCA}.
Since there was no clue how to choose the six parameters $\rho$ in the six problems \eqref{prob:PCA-dAspremont-sdp-pen-L1} when DSPCA is used to solve them, we did not compare with DSPCA for solving \eqref{prob:PCA-dAspremont-sdp-pen-L1}. From Tables
\ref{tab:pitprops-ADMM} and \ref{tab:pitprops-ALSPCA} we see that both ADMM and ALSPCA gave a solution with 15 nonzeros in the
first six sparse PCs, and the solution given by ADMM explains slightly more variance than the solution given by ALSPCA.

\begin{table}[ht]{
\begin{center}\caption{First six sparse PCs of the pit props data set given by ADMM}\label{tab:pitprops-ADMM}
\begin{tabular}{|l|llllll|}\hline

Variables & PC1 & PC2 & PC3 & PC4 & PC5 & PC6 \\\hline

Topdiam   & -0.4908 & 0       & 0      &   0       & 0       & 0       \\\hline
Length    & -0.5067 & 0       & 0      &   0       & 0       & 0       \\\hline
Moist     & 0       & -0.7175 & 0      &   0       & 0       & 0       \\\hline
Testsg    & 0       & -0.6965 & 0      &   0       & 0       & 0       \\\hline
Ovensg    & 0       &  0      & 0.9263 &   0       & 0       & 0       \\\hline
Ringtop   & -0.0668 &  0      & 0.3511 &   0       & 0       & 0       \\\hline
Ringbut   & -0.3565 &  0      & 0.1369 &   0       & 0       & 0       \\\hline
Bowmax    & -0.2334 &  0      & 0      &   0       & 0       & 0       \\\hline
Bowdist   & -0.3861 &  0      & 0      &   0       & 0       & 0       \\\hline
Whorls    & -0.4089 &  0      & 0      &   0       & 0       & 0       \\\hline
Clear     & 0       &  0      & 0      &   1.0000  & 0       & 0       \\\hline
Knots     & 0       &  0      & 0      &   0       & 1.0000  & 0       \\\hline
Diaknot   & 0       &  0      & 0      &   0       & 0       & 1.0000  \\\hline

\multicolumn{7}{|c|}{Total sparsity: 15, total explained variance: $74.31\%$} \\\hline
\end{tabular}
\end{center}}
\end{table}

\begin{table}[ht]{
\begin{center}\caption{First six sparse PCs of the pit props data set given by ALSPCA}\label{tab:pitprops-ALSPCA}
\begin{tabular}{|l|llllll|}\hline

Variables & PC1 & PC2 & PC3 & PC4 & PC5 & PC6 \\\hline

Topdiam   & 0.4052   & 0        & 0        & 0        & 0         & 0 \\\hline
Length    & 0.4248   & 0        & 0        & 0        & 0         & 0 \\\hline
Moist     & 0        & -0.7262  & 0        & 0        & 0         & 0 \\\hline
Testsg    & 0.0014   & -0.6874  & 0        & 0        & 0         & 0 \\\hline
Ovensg    & 0        & 0        & -1.0000  & 0        & 0         & 0 \\\hline
Ringtop   & 0.1857   & 0        & 0        & 0        & 0         & 0 \\\hline
Ringbut   & 0.4122   & 0        & 0        & 0        & 0         & 0 \\\hline
Bowmax    & 0.3277   & 0        & 0        & 0        & 0         & 0 \\\hline
Bowdist   & 0.3829   & 0        & 0        & 0        & 0         & 0 \\\hline
Whorls    & 0.4437   & 0.0021   & 0        & 0        & 0         & 0 \\\hline
Clear     & 0        & 0        & 0        & 1.0000   & 0         & 0 \\\hline
Knots     & 0        & 0        & 0        & 0        & 1.0000    & 0 \\\hline
Diaknot   & 0        & 0        & 0        & 0        & 0         & 1.0000 \\\hline

\multicolumn{7}{|c|}{Total sparsity: 15, total explained variance: $73.29\%$} \\\hline
\end{tabular}
\end{center}}
\end{table}

\subsection{Random examples}

We created some random examples to test the speed of
ADMM and compared it with DSPCA \cite{daspremont-sparsePCA-direct-formulation-2007} and ALSPCA \cite{Lu-Zhang-sparsePCA-MPA-2011}.
The sample covariance matrix $\Sigma$ was created by adding some small noise to a sparse rank-one matrix. Specifically, we first created a sparse vector
$\hat{x}\in\br^p$ with $s$ nonzeros randomly chosen from the Gaussian distribution $\mathcal{N}(0,1)$.
We then got the sample covariance matrix $\Sigma = \hat{x}\hat{x}^\top+\sigma vv^\top$,
where $\sigma$ denotes the noise level and $v\in\br^p$ is a random vector with entries uniformly drawn from $[0,1]$. We applied DSPCA, ALSPCA and ADMM to find the
largest sparse PC of $\Sigma$.
We report the comparison results in Tables \ref{tab:random-compare-1} and \ref{tab:random-compare-2} that correspond to noise levels $\sigma=0.01$ and $\sigma=0.1$
respectively. When using DSPCA to solve \eqref{prob:PCA-dAspremont-sdp-pen-L1} and ALSPCA to solve \eqref{prob:PCA-zhaosong},
we set different $\rho$'s to get solutions with different sparsity levels. Specifically, we tested DSPCA for $\rho=0.01,0.1,1$ in Table \ref{tab:random-compare-1}
with $\sigma=0.01$ and $\rho=0.1,1,10$ in Table \ref{tab:random-compare-2} with $\sigma=0.1$; we tested ALSPCA for $\rho=0.01,0.1,1$ in both
Tables \ref{tab:random-compare-1} and \ref{tab:random-compare-2}. We set different $K$'s in \eqref{prob:PCA-dAspremont-sdp-con-L1} to control the sparsity level when
using ADMM to solve it. In both Tables \ref{tab:random-compare-1} and \ref{tab:random-compare-2}, we tested four data sets with dimension $p$ and sparsity $s$
setting as $(p,s)=(100,10),(100,20),(200,10)$ and $(200,20)$. We used the following continuation technique for $\mu$ in ADMM:
$\mu_0=1, \mu_{k}=\max\{2\mu_{k-1}/3,10^{-4}\}.$ $\Delta_{ij}$ were set to 0.1 for all $i\neq j$ in all the tests for ALSPCA as suggested in
\cite{Lu-Zhang-sparsePCA-MPA-2011} for tests on random data sets.

We report the cardinality of the largest sparse PC (Card), the percentage of the explained variance (PEV) and the CPU time in Tables \ref{tab:random-compare-1}
and \ref{tab:random-compare-2}. From Table \ref{tab:random-compare-1} we see that, for $\sigma=0.01$, all three algorithms DSPCA, ADMM and ALSPCA are sensitive to
the parameters that control the sparsity, i.e., $\rho$ and $K$. $\rho=0.01$ always gave the best results for DSPCA and ALSPCA and the explained variance is very close
to the standard PCA. $\rho=0.1$ still provided relatively good solutions for DSPCA and ALSPCA in terms of both sparsity and the explained variance. When $\rho$ was
increased to 1, the solutions given by ALSPCA were too sparse to give a relatively large explained variance; while the solutions given by DSPCA sometimes had more nonzeros than the desired sparsity level (when $(p,s)=(100,10)$), and even when the solutions were of the desired sparsity level, the explained variances were affected a lot (when $(p,s)=(100,20)$ and $(200,20)$). For ADMM, $K=5,4,3$ were tested
for $s=10$ and $K=10,9,8$ were tested for $s=20$. Results shown in Table \ref{tab:random-compare-1} indicate that $K=s/2$ usually produced good results.
When $K$ was changed from $5$ to $4$ and $3$ for $s=10$, the sparsity and explained variance of the solution changed a lot. When $K$ was changed
from $10$ to $9$ and $8$ for $s=20$, the solution was not affected too much in terms of the explained variance. Especially, for $(p,s)=(100,20)$ and $(200,20)$ and $K=8$, ADMM gave solutions with sparsity $13$ that explain $96.10\%$ and $93.66\%$ variance respectively, which are both very close to the results given by the standard PCA. From Table \ref{tab:random-compare-2} we see that,
for $\sigma=0.1$, i.e., when the noise level was large, DSPCA and ALSPCA were more sensitive to the noise compared with their performance when $\sigma=0.01$.
More specifically, in Table
\ref{tab:random-compare-2}, $\rho=0.1$ usually gave a solution with the best explained variance and appropriate sparsity for DSPCA, expect for $(p,s)=(100,10)$, where the solution produced by DSPCA had 21 nonzeros, which was much more than the desired sparsity 10. The solutions given by DSPCA for
$\rho=1$ and $\rho=10$ were not very satisfied. For ALSPCA, when $(p,s)=(100,10)$ and $(100,20)$, $\rho=0.1$ gave good results, while the results given by
$\rho=0.01$ and $\rho=1$ were not very satisfied. However, we observed that the
performance of ADMM when $\sigma=0.1$ was consistent with its performance when $\sigma=0.01$, i.e., its performance was not very sensitive to the noise.

From both Tables \ref{tab:random-compare-1} and
\ref{tab:random-compare-2}, we see that ALSPCA was slightly faster than ADMM, and they were both significantly faster than DSPCA. This is reasonable because ALSPCA solves the non-convex problem \eqref{prob:PCA-zhaosong} and thus eigenvalue decomposition is not required, which costs most of the computational effort in DSPCA and ADMM.

\begin{table}[ht]{
\begin{center}\caption{Comparisons of ADMM, ALSPCA and DSPCA on random examples with $\sigma=0.01$}\label{tab:random-compare-1}
\begin{tabular}{|l|l|l|l|l|l|}\hline
$(p,s)$ & Method & Parameters    & Card & PEV & CPU \\\hline
(100,10)   & PCA    &               &    & 96.16\% & \\
           & DSPCA  & $\rho=0.01$   & 10 & 96.16\% & 12.12 \\
           &        & $\rho=0.1$    & 10 & 95.81\% & 8.29 \\
           &        & $\rho=1$      & 13 & 87.28\% & 6.56 \\
           & ADMM   & $K=5$         & 9  & 95.30\% & 0.26 \\
           &        & $K=4$         & 6  & 91.55\% & 0.28 \\
           &        & $K=3$         & 4  & 79.30\% & 0.19 \\
           & ALSPCA & $\rho=0.01$   & 10 & 96.16\% & 0.13 \\
           &        & $\rho=0.1$    & 9  & 96.02\% & 0.39 \\
           &        & $\rho=1$      & 4  & 89.54\% & 0.13 \\\hline
(100,20)   & PCA    &               &    & 98.07\% &       \\
           & DSPCA  & $\rho=0.01$   & 20 & 98.07\% & 10.93 \\
           &        & $\rho=0.1$    & 20 & 97.71\% & 8.47  \\
           &        & $\rho=1$      & 20 & 85.25\% & 5.75  \\
           & ADMM   & $K=10$        & 20 & 97.98\% & 0.28  \\
           &        & $K=9$         & 18 & 97.40\% & 0.28  \\
           &        & $K=8$         & 13 & 96.10\% & 0.31  \\
           & ALSPCA & $\rho=0.01$   & 20 & 98.07\% & 0.13 \\
           &        & $\rho=0.1$    & 18 & 97.83\% & 0.13 \\
           &        & $\rho=1$      & 8  & 83.87\% & 0.13 \\\hline
(200,10)   & PCA    &               &    & 91.43\% &       \\
           & DSPCA  & $\rho=0.01$   & 10 & 91.42\% & 88.87 \\
           &        & $\rho=0.1$    & 10 & 91.09\% & 61.36 \\
           &        & $\rho=1$      & 8  & 82.91\% & 45.97 \\
           & ADMM   & $K=5$         & 9  & 90.61\% & 1.23 \\
           &        & $K=4$         & 6  & 87.04\% & 1.27 \\
           &        & $K=3$         & 4  & 75.40\% & 0.88 \\
           & ALSPCA & $\rho=0.01$   & 10 & 91.42\% & 0.23 \\
           &        & $\rho=0.1$    & 9  & 91.29\% & 0.23 \\
           &        & $\rho=1$      & 4  & 85.13\% & 0.29 \\\hline
(200,20)   & PCA    &               &    & 95.58\% &      \\
           & DSPCA  & $\rho=0.01$   & 20 & 95.58\% & 79.87 \\
           &        & $\rho=0.1$    & 20 & 95.22\% & 63.12 \\
           &        & $\rho=1$      & 20 & 83.09\% & 42.22 \\
           & ADMM   & $K=10$        & 20 & 95.49\% & 1.57 \\
           &        & $K=9$         & 18 & 94.93\% & 1.67 \\
           &        & $K=8$         & 13 & 93.66\% & 1.71 \\
           & ALSPCA & $\rho=0.01$   & 20 & 95.58\% & 0.23 \\
           &        & $\rho=0.1$    & 18 & 95.34\% & 0.23 \\
           &        & $\rho=1$      & 8  & 81.73\% & 0.23 \\\hline
\end{tabular}
\end{center}}
\end{table}

\begin{table}[ht]{
\begin{center}\caption{Comparisons of ADMM, ALSPCA and DSPCA on random examples with $\sigma=0.1$}\label{tab:random-compare-2}
\begin{tabular}{|l|l|l|l|l|l|}\hline
$(p,s)$ & Method & Parameters    & Card & PEV & CPU \\\hline
(100,10)   & PCA    &               &    & 71.51\% &      \\
           & DSPCA  & $\rho=0.1$    & 21 & 71.23\% & 9.42 \\
           &        & $\rho=1$      & 10 & 64.92\% & 4.40 \\
           &        & $\rho=10$     & 1  & 27.04\% & 4.14 \\
           & ADMM   & $K=5$         & 9  & 70.83\% & 0.24 \\
           &        & $K=4$         & 6  & 68.03\% & 0.25 \\
           &        & $K=3$         & 4  & 58.93\% & 0.17 \\
           & ALSPCA & $\rho=0.01$   & 31 & 71.49\% & 0.13 \\
           &        & $\rho=0.1$    & 9  & 71.39\% & 0.13 \\
           &        & $\rho=1$      & 4  & 66.53\% & 0.13 \\\hline
(100,20)   & PCA    &               &    & 83.59\% &      \\
           & DSPCA  & $\rho=0.1$    & 20 & 83.27\% & 8.58 \\
           &        & $\rho=1$      & 20 & 72.75\% & 5.82 \\
           &        & $\rho=10$     & 56 & 26.80\% & 3.01 \\
           & ADMM   & $K=10$        & 20 & 83.50\% & 0.32 \\
           &        & $K=9$         & 18 & 83.02\% & 0.40 \\
           &        & $K=8$         & 13 & 81.90\% & 0.27 \\
           & ALSPCA & $\rho=0.01$   & 25 & 83.58\% & 0.13 \\
           &        & $\rho=0.1$    & 19 & 83.37\% & 0.13 \\
           &        & $\rho=1$      & 8  & 71.37\% & 0.13 \\\hline
(200,10)   & PCA    &               &    & 51.69\% &      \\
           & DSPCA  & $\rho=0.1$    & 10 & 51.46\% & 19.19 \\
           &        & $\rho=1$      & 88 & 46.95\% & 28.51 \\
           &        & $\rho=10$     & 1  & 19.80\% & 28.04 \\
           & ADMM   & $K=5$         & 9  & 51.15\% & 1.22 \\
           &        & $K=4$         & 6  & 49.13\% & 1.38 \\
           &        & $K=3$         & 4  & 42.61\% & 0.87 \\
           & ALSPCA & $\rho=0.01$   & 10 & 51.61\% & 0.26 \\
           &        & $\rho=0.1$    & 9  & 51.53\% & 0.25 \\
           &        & $\rho=1$      & 4  & 48.01\% & 0.24 \\\hline
(200,20)   & PCA    &               &    & 68.38\% &      \\
           & DSPCA  & $\rho=0.1$    & 20 & 68.12\% & 74.87 \\
           &        & $\rho=1$      & 20 & 59.54\% & 42.63 \\
           &        & $\rho=10$     & 64 & 22.03\% & 34.83 \\
           & ADMM   & $K=9$         & 18 & 67.91\% & 1.46 \\
           &        & $K=8$         & 14 & 67.01\% & 1.74 \\
           &        & $K=7$         & 11 & 65.26\% & 1.76 \\
           & ALSPCA & $\rho=0.01$   & 20 & 68.37\% & 0.24 \\
           &        & $\rho=0.1$    & 18 & 68.20\% & 0.25 \\
           &        & $\rho=1$      & 8  & 58.37\% & 0.24 \\\hline
\end{tabular}
\end{center}}
\end{table}

\subsection{Text data classification}

Sparse PCA can also be used to classify the keywords in text data. This application has been studied by Zhang, d'Aspremont and El Ghaoui in
\cite{Zhang-daspremont-sparsePCA-handbook-2011} and Zhang and El Ghaoui in \cite{Zhang-ElGhaoui-sparsePCA-NIPS-2011}.
In this section, we show that by using our ADMM to solve the sparse PCA problem, we can also classify the keywords from text data very well.
The data set we used is a small version of the ``20-newsgroups'' data\footnote{This data set can be downloaded from http://cs.nyu.edu/$\sim$roweis/data.html.},
which is also used in \cite{Zhang-daspremont-sparsePCA-handbook-2011}.
This data set consists of the binary occurrences of 100 specific words across 16242 postings,
i.e., the data matrix $M$ is of the size $100\times 16242$ and $M_{ij}=1$ if the $i$-th word appears at least once
in the $j$-th posting and $M_{ij}=0$ if the $i$-th word does not appear in the $j$-th posting.
These words can be approximately divided into different groups such as ``computer'', ``religion'' etc. We want to find the words
that contribute as much variance as possible and also discover which words are in the same category. By viewing each posting
as a sample of the 100 variables, we have 16424 samples of the variables, and thus the sample covariance matrix $\Sigma\in\br^{100\times 100}$.
Using standard PCA, it is hard to interpret which words contribute to each of the leading eigenvalues since the loadings are dense.
However, sparse
PCA can explain as much the variance explained by the standard PCs, and meanwhile interpret well which words contribute together
to the corresponding variance. We applied our ADMM to solve \eqref{prob:PCA-dAspremont-sdp-con-L1} to find the first three sparse PCs.
We set $K=5$ in all three problems and the following continuation technique was used for $\mu$:
$\mu_0=100, \mu_{k} = \max\{2\mu_{k-1}/3,10^{-4}\}$. The resulting three sparse PCs have 10, 12 and 17 nonzeros respectively. The total explained variance
by these three sparse PCs is $12.72\%$, while the variance explained by the largest three PCs by the standard PCA is $19.10\%$.

The words corresponding to the first three sparse PCs generated by our ADMM
are listed in Table \ref{tab:20news}.
From Table \ref{tab:20news} we see that the words in the first sparse PC are approximately in the category ``school'',
the words in the second PC are approximately in the category ``religion'', and the words in the third sparse PC are approximately in the category ``computer''.
So our ADMM can classify the keywords into appropriate categories very well.

\begin{table}[ht]{
\begin{center}\caption{Words associated with the first three sparse PCs using ADMM}\label{tab:20news}
\begin{tabular}{|l|l|l|}\hline

1st PC (10 words)        & 2nd PC (12 words)                 & 3rd PC (17 words)                   \\\hline

case         & bible                 & computer     \\

course       & case                  & email     \\

email        & christian             & files     \\

fact         & course                & ftp     \\

help         & evidence              & graphics     \\

number       & fact                  & number     \\

problem      & god                   & phone     \\

question     & government            & problem      \\

system       & human                 & program      \\

university   & jesus                 & research       \\

             & religion              & science       \\

             & world                 & software       \\

             &                         & space         \\

             &                         & state      \\

             &                         & university      \\

             &                         & version       \\

             &                         & windows      \\ \hline

\multicolumn{3}{|c|}{Total sparsity: 39, total explained variance: $12.72\%$} \\\hline
\end{tabular}
\end{center}}
\end{table}

\subsection{Senate voting data}

In this section, we use sparse PCA to analyze the voting records of the 109th US Senate, which was also studied by Zhang, d'Aspremont and El Ghaoui in
\cite{Zhang-daspremont-sparsePCA-handbook-2011}. The votes are recorded as $1$ for ``yes'' and $-1$ for ``no''.
Missing votes are recorded as 0. There are 100 senators (55 Republican, 44 Democratic and 1 independent)
and 542 bills involved in the data set. However, there are many missing votes in the data set. To obtain a meaningful data matrix, we only choose
the bills for which the number of missing votes is at most one.
There are only 66 such bills among the 542 bills. So our data matrix $M$ is a $66\times 100$ matrix with entries $1$, $-1$ and $0$,
and each column of $M$ corresponds to one senator's voting. The
sample covariance matrix $\Sigma=MM^\top$ in our test is a $66\times 66$ matrix.

To see how standard PCA and sparse PCA perform in classifying the voting records, we implemented the following
procedure as suggested in \cite{Zhang-daspremont-sparsePCA-handbook-2011}. We used standard PCA to find the largest two PCs (denoted as $v_1$ and $v_2$)
of $\Sigma$. We then projected each column of $M$ onto the subspace spanned by $v_1$ and $v_2$, i.e., we found $\bar{\alpha}_i$ and $\bar{\beta}_i$
for each column $M_i$ such that
\[(\bar{\alpha}_i,\bar{\beta}_i):= \arg\min_{(\alpha_i,\beta_i)} \|\alpha_i v_1+\beta_i v_2 - M_i\|.\]
We then drew each column $M_i$ as a point $(\bar{\alpha}_i,\bar{\beta}_i)$ in the two-dimensional subspace spanned by $v_1$ and $v_2$. The left figure in Figure
\ref{fig:senate-2D-projection} shows the $100$ points. We see from this figure that senators are separated very well by partisanship. However, it is hard to
interpret which bills are responsible to the explained variance, because all the bills are involved in the PCs. By using sparse PCA, we can interpret the
explained variance by just a few bills. We applied our ADMM to find the first two sparse PCs (denoted as $s_1$ and $s_2$) of $\Sigma$.
We set $K=4$ for both problems and used
the following continuation technique on $\mu$:
$\mu_0=100, \mu_{k} = \max\{2\mu_{k-1}/3,10^{-4}\}.$

The resulting two sparse PCs $s_1$ and $s_2$ produced by our ADMM have 9 and 5 nonzeros respectively. We projected each column of $M$ onto the subspace
spanned by these two sparse PCs. The right figure in Figure \ref{fig:senate-2D-projection} shows the $100$ projections onto the subspace spanned by the sparse PCs $s_1$ and $s_2$.
We see from this
figure that the senators are still separated well by partisanship. Now since only a few bills are involved in the two sparse PCs, we can interpret which bills
are responsible most for the classification.
The bills involved in the first two PCs are listed in Table \ref{tab:senate-two-bills}. From Table \ref{tab:senate-two-bills}
we see that the most controversial issues between Republicans and Democrats are topics such as ``Budget'' and ``Appropriations''. Other controversial issues
involve topics like ``Energy'', ``Abortion'' and ``Health''.

\begin{table}[ht]{
\begin{center}\caption{Bills involved in the first two PCs by ADMM}\label{tab:senate-two-bills}
\begin{tabular}{|l|}\hline
\multicolumn{1}{|c|}{Bills in the first sparse PC} \\\hline
Budget, Spending and Taxes\_Education Funding Amendment\_3804 \\
Budget, Spending and Taxes\_Reinstate Pay-As-You-Go through 2011 Amendment\_3806 \\
Energy Issues\_LIHEAP Funding Amendment\_3808 \\
Abortion Issues\_Unintended Pregnancy Amendment\_3489 \\
Budget, Spending and Taxes\_Budget FY2006 Appropriations Resolution\_3488 \\
Budget, Spending and Taxes\_Budget Reconciliation bill\_3665 \\
Budget, Spending and Taxes\_Budget Reconciliation bill\_3789 \\
Budget, Spending and Taxes\_Education Amendment\_3490 \\
Health Issues\_Medicaid Amendment\_3496 \\\hline
%
\multicolumn{1}{|c|}{Bills in the second sparse PC} \\\hline
Appropriations\_Agriculture, Rural Development, FDA Appropriations Act\_3677 \\
Appropriations\_Emergency Supplemental Appropriations Act, 2005\_3515 \\
Appropriations\_Emergency Supplemental Appropriations Act, 2006\_3845 \\
Appropriations\_Interior Department FY 2006 Appropriations Bill\_3595 \\
Executive Branch\_John Negroponte, Director of National Intelligence\_3505 \\\hline

\end{tabular}
\end{center}}
\end{table}

\begin{figure}
\centering \subfigure{\includegraphics[scale=0.5]{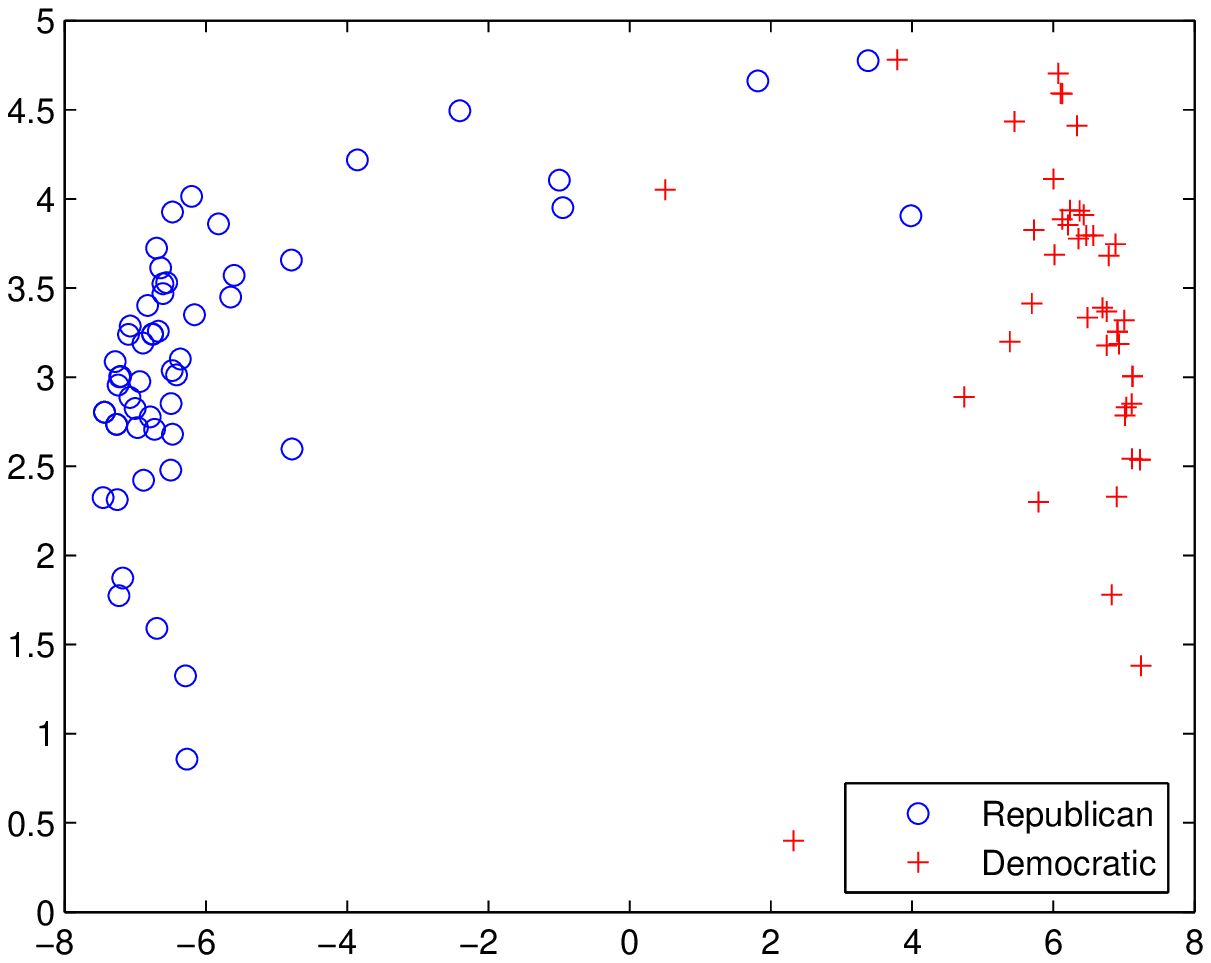}}\hspace{-0.5cm}
\centering \subfigure{\includegraphics[scale=0.5]{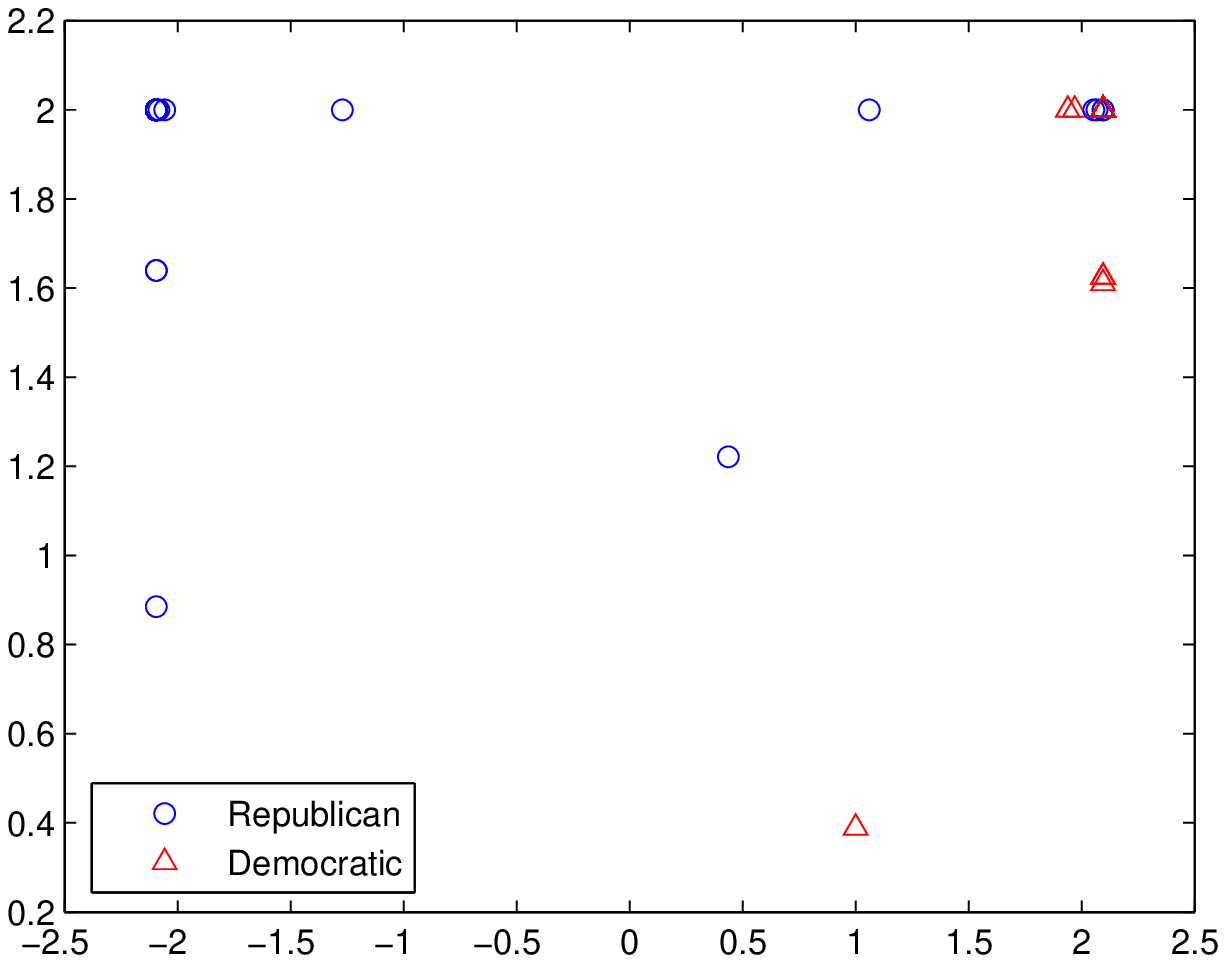}} \hspace{-0.5cm}
\caption{Projection of the senate voting records onto the subspace spanned by the
top 2 principal components: Left: standard PCA; Right: sparse PCA} \label{fig:senate-2D-projection}
\end{figure}

\section{Conclusion}\label{sec:conclusion}

In this paper, we proposed alternating direction method of multipliers to solve an SDP relaxation of the sparse PCA problem. Our method incorporated
a variable-splitting technique to separate the $\ell_1$ norm constraint, which controls the sparsity of the solution, and the positive-semidefiniteness constraint. This method
resulted in two relatively simple subproblems that have closed-form solutions in each iteration. Global convergence results were established for the proposed method.
Numerical results on both synthetic data and real data from classification of text data and senate voting records demonstrated the efficacy of our method.

Compared with Nesterov's first-order method DSPCA for sparse PCA studied in \cite{daspremont-sparsePCA-direct-formulation-2007}, our ADMM method solves the primal problems directly and guarantees sparse solutions. Numerical results also indicate that ADMM is much faster than DSPCA. Compared with methods for solving nonconvex formulations of sparse PCA, the nonsmooth SDP formulation considered in this paper usually requires more computational
effort in each iteration. However, the global convergence of our ADMM for solving the nonsmooth SDP is guaranteed, while methods for solving nonconvex problems
usually have only local convergence.

\section*{Acknowledgement}
The author is grateful to Alexandre d'Aspremont for discussions on using DSPCA. The author thanks Stephen J. Wright and Lingzhou Xue for reading an earlier
version of the manuscript and for helpful discussions. This work was partially supported by the NSF Postdoctoral Fellowship through the
Institute for Mathematics and Its Applications at University of Minnesota.

\bibliographystyle{siam}
\bibliography{C:/Mywork/Optimization/work/reports/bibfiles/All}

\end{document}